\documentclass[12pt]{amsart}
\usepackage{amsmath}
\usepackage{latexsym}
\usepackage{amssymb}
\usepackage{graphicx}
\DeclareMathOperator{\Char}{Char}

\DeclareMathOperator{\dist}{dist}

\DeclareMathOperator{\sing}{sing}
\DeclareMathOperator{\Span}{span}

\newtheorem{theorem}{Theorem}
\newtheorem{proposition}{Proposition}

\newtheorem{corollary}{Corollary}
\newtheorem{definition}{Definition}
\newtheorem{example}{Example}

\newtheorem{remark}{Remark}

\begin{document}
%
% \baselineskip 18pt
%%%%%%%%%%%%%%%%%%%%%%%%%%%%%%%%%%%%%%%%%%%%%%%%%%%%%%%%%%%%%%%%%%%%%%
%%%%%%%%%%%%%%%% Blackboard boldface common symbols  %%%%%%%%%%%%%%%%%
%%%%%%%%%%%%%%%%%%%%%%%%%%%%%%%%%%%%%%%%%%%%%%%%%%%%%%%%%%%%%%%%%%%%%%
\def\R {{\mathbb{R}}}
\def\N {{\mathbb{N}}}
\def\C {{\mathbb{C}}}
\def\Z {{\mathbb{Z}}}
%%%%%%%%%%%%%%%%%%%%%%%%%%%%%%%%%%%%%%%%%%%%%%%%%%%%%%%%%%%%%%%%%%%%%%
%%%%%%%%%%%%%%%%%% varsymbols ... for readability %%%%%%%%%%%%%%%%%%%%
%%%%%%%%%%%%%%%%%%%%%%%%%%%%%%%%%%%%%%%%%%%%%%%%%%%%%%%%%%%%%%%%%%%%%%
\def\phi{\varphi}
\def\epsilon{\varepsilon}
\def\O{\Omega}
\def\bO{\overline{\O}}
\def\hp{{\rm hypo}}
\def\hps{\partial^{P,\infty}}
\def\ox{\overline{x}}
%%%%%%%%%%%%%%%%%%%%%%%%%%%%%%%%%%%%%%%%%%%%%%%%%%%%%%%%%%%%%%%%%%%%%%
%%%%%%%%%%%%%%%%%%%%%%%%%%%%%%%%%%%%%%%%%%%%%%%%%%%%%%%%%%%%%%%%%%%%%%
%
%%%%%%%%%%%%%%%%%%%%%%%%%%%%%%%%%%%%%%%%%%%%%%%%%%%%%%%%%%%%%%%%%%%%%%
%%%%%%%%%%%%%%%%%%%%  math local macros %%%%%%%%%%%%%%%%%%%%%%%%%%%%%%
%%%%%%%%%%%%%%%%%%%%%%%%%%%%%%%%%%%%%%%%%%%%%%%%%%%%%%%%%%%%%%%%%%%%%%
\def\tb#1{\|\kern -1.2pt | #1 \|\kern -1.2pt |} 
\def\Qed{\qed\par\medskip\noindent}
%%%%%%%%%%%%%%%%%%%%%%%%%%%%%%%%%%%%%%%%%%%%%%%%%%%%%%%%%%%%%%%%%%%%%%
%%%%%%%%%%%%%%%%%%%%%%%%%%%%%%%%%%%%%%%%%%%%%%%%%%%%%%%%%%%%%%%%%%%%%%
%
\title[The minimum time function with H\"ormander vector fields]{Regularity results for the minimum time function with H\"ormander vector fields} 
\thanks{This work was started during the ``Workshop on Hamilton-Jacobi equations''  held at Fudan University, Shanghai, China (July 24-30, 2016). The first author was supported by RFO grant, Universit\`a di Bologna, Italy. 
The second author was partly supported by the University of Rome Tor Vergata 
(Consolidate the Foundations 2015) and the Istituto Nazionale di Alta Matematica ``F. Severi'' (GNAMPA 2016 Research Projects). 
The third author gratefully acknowledges the financial support of the ``Vienna Graduate School on Computational Optimization'' which is funded by Austrian Science Foundation (FWF, project no. W1260-N35).}
\author{Paolo Albano} 
\address{Dipartimento di Matematica, 
Universit\`a di Bologna, Piazza
di Porta San Donato 5, 40127 Bologna, Italy} 
\email{paolo.albano@unibo.it}
\author{Piermarco Cannarsa}
\address{Dipartimento di Matematica, 
Universit\`a di Roma "Tor Vergata", Via della Ricerca Scientifica 1, 00133 Roma, Italy\footnote{Corresponding Author}}
\email{cannarsa@mat.uniroma2.it} 
\author{Teresa Scarinci}
\address{Department of Statistics and Operation Research, University of Vienna, Oskar-Morgenstern-Platz 1,
1090 Vienna, Austria}
\email{teresa.scarinci@gmail.com}
\date{\today}

\begin{abstract}
In a bounded domain of  $\R^n$ with smooth boundary, we study the regularity of the viscosity solution, $T$,
of the Dirichlet problem for the eikonal equation associated with a family of smooth vector fields $\{X_1,\ldots ,X_N\}$, subject to H\"ormander's 
bracket generating condition.  Due to the presence of characteristic boundary points, singular trajectories may
occur in this case. We characterize such trajectories as the closed set of all points at which the solution loses point-wise 
Lipschitz  continuity. We then prove that the local Lipschitz continuity of $T$, the local semiconcavity of $T$, and
the absence of singular trajectories are equivalent properties.  Finally, we show that the last condition is satisfied when the
characteristic set of $\{X_1,\ldots ,X_N\}$ is a symplectic manifold. We apply our results to Heisenberg's and Martinet's vector fields.
\end{abstract}


\maketitle

\section{Introduction}
\setcounter{equation}{0}
\setcounter{theorem}{0}
\setcounter{proposition}{0}  
\setcounter{lemma}{0}
\setcounter{corollary}{0} 
\setcounter{definition}{0}
In a bounded open set $\Omega\subset\R^n$ with smooth boundary, $\Gamma$, we study the regularity of the viscosity solution, $T$,
to the Dirichlet problem 
 \begin{equation}
\label{intro:sosq}
\left\{
\begin{array}{l}
\sum_{j=1}^N (X_jT)^2(x)=1\quad \text{ in
}\quad \Omega 
\vspace{.2cm}
\\
T=0 \quad \text{ on
}\quad \Gamma ,
\end{array}
\right . 
\end{equation}
where $\{X_1,\ldots ,X_N\}$ is a system of smooth vector fields which satisfies
     H\"ormander's bracket generating condition (\cite{H}). 
 It is well known that such a solution may be interpreted as 
the sub-Riemannian distance from the boundary of $\Omega$ when $\{X_1,\ldots ,X_N\}$ span a distribution 
or, in control theory, as a certain minimum time function. 
While structural properties of the singular set and the regularity of the Riemannian distance function have been widely investigated (see, e.g., \cite{A0}, \cite{CMN},  \cite{CPS},  \cite{CSi1},  \cite{CSi2}), similar issues have not been much addressed for the solution of \eqref{eq:sosq} in the sub-Riemannian case or, more generally, when the quadratic form associated with the eikonal equation in \eqref{eq:sosq} fails to be positive definite (see, e.g., \cite{A1}).
The analysis of such questions is the main purpose  of this paper. 

In our approach, we will often  use  the characterization of the unique solution, $T$, of \eqref{eq:sosq} as the value function of the time optimal control problem with target $\Gamma$ and state equation
 \begin{equation}\label{intro:se}
 \begin{cases}
 y'(t)=\sum_{j=1}^N u_j(t)X_j(y(t))
 &
 (t\geq 0)
 \\
 y(0)=x,
 \end{cases}
\end{equation}
where $u:[0,+\infty[\to \overline{B}_1(0)$. In particular, a crucial role in our investigation will be played by the so-called {\em singular time-optimal trajectories} of \eqref{intro:se}, a well-known object in geometric control theory (see, e.g., \cite{CJT}). Singular trajectories are time-optimal trajectories that can be completed to a solution of the 
characteristic system associated with \eqref{intro:sosq}, which belongs to the {\em characteristic set}  of $\{X_1,\ldots ,X_N\}$  and satisfies a suitable transversality condition (see Definition~\ref{optimal_triple} below). Equivalently, singular time-optimal trajectories can be identified as those time-optimal trajectories that hit $\Gamma$ at a {\em characteristic point}, that is, a point at which $X_1,\ldots ,X_N$ are tangent to $\Gamma$ (see Theorem~\ref{t:sincar} below).

The fact that the existence of singular trajectories may destroy the smoothness (subanalyticity) of a solution of a first order Hamilton-Jacobi equation
was already observed in \cite{S}, \cite{Ag}, and \cite{Tre} (see also Remark~\ref{re:Sect4} below). In this paper,  we interested in a more basic smoothness threshold---namely Lipschitz continuity---which, as we explain below, opens the way to higher regularity properties for the solution of \eqref{eq:sosq}. 
Therefore, we begin our analysis by giving a {\em necessary and sufficient condition} for point-wise Lipschitz continuity, showing that $T$ fails to be Lipschitz continuous at a point $x\in \O$  if and only if $x$ is the starting point of a singular time-optimal trajectory (Theorem~\ref{t:nls}). By dynamic programming, this result implies that the set of all singular time-optimal trajectories consists---as a point set---of all the points of $\bO$ at which $T$ fails to be Lipschitz (Corollary~\ref{t:lipschitz-no-sing}).  

We note that point-wise Lipschitz continuity is a very weak property. Nevertheless, Theorem~\ref{t:nls} can be used to derive a classical Lipschitz regularity result whenever one can exclude the presence of singular time-optimal trajectories. This is the viewpoint leading to Theorem~\ref{t:semiconcave}, which ensures the equivalence of  the following  three properties: 
\begin{itemize}
\item[(a)] system \eqref{intro:se} admits no singular time-optimal trajectory;
\item[(b)] $T$ locally is semiconcave in $\Omega$;
\item[(c)] $T$ is locally Lipschitz in $\Omega$.
\end{itemize}
We observe that, in \cite{CR}, the local semiconcavity of the sub-Riemannian distance to a point $x_0$, $d_{SR}(x_0,\cdot)$, is proved assuming the absence of singular time-optimal trajectories, but without giving any  estimate of the dependence on $x_0$ of the semiconcavity constant of $d_{SR}(x_0,\cdot)$. Thus, the local semiconcavity of the minimum time function for a {\em general} target (i.e., (a) $\Rightarrow$ (b)) is not a direct consequence of the result of \cite{CR}, even when  $\{X_1,\ldots ,X_N\}$ generates a distribution.

We conclude the introduction by quickly mentioning further results that complete our exposition, while outlining the structure of this paper. 
In Section~\ref{Sect:vfmt}, we recall all required notions concerning H\"ormander vector fields, characteristic sets,  time optimal control problems, and  viscosity solutions of  related eikonal equations.
In Section~\ref{Sect:stot}, we introduce singular time-optimal trajectories and develop our analysis of the behaviour of $T$ along such trajectories.
Based on these results, in Section~\ref{Sect:reg}, we obtain the aforementioned characterization of the local Lipschitz regularity of $T$, we give sufficient conditions on the characteristic set of $\{X_1,\ldots ,X_N\}$  in order to exclude the presence of singular trajectories, and we discuss the examples of Heisenberg's and Martinet's vector fields as applications of our results. We conclude the paper with Appendix~\ref{a:b}, where we prove a regularity result for $T$  at the boundary of $\O$  (Theorem~\ref{t:bregularity}), which might be known to the reader in other forms than the one which fits the analysis of this paper.

\section{H\"ormander vector fields and minimum time function}\label{Sect:vfmt}
Let $\Omega \subset\R^n$ be a bounded open  set and assume that the
boundary of $\Omega$, $\Gamma$, is a smooth manifold of dimension   
$n-1$. 

%\medskip
We denote by $VF(\O)$ the  space of all $C^\infty$ vector fields on $\O$.
Given any two such vector fields,
\begin{equation*}
X(x)=\sum_{i=1}^nf_i(x)\partial_{x_i}\quad\mbox{and}\quad Y(x)=\sum_{i=1}^ng_i(x)\partial_{x_i}\qquad(x\in\O),
\end{equation*}
with $f_i,g_i\in C^\infty(\O)\,(i=1,\dots,n)$,  we denote by $[X,Y]$ the Lie bracket
\begin{equation*}
[X,Y](x)=\sum_{i=1}^nh_i(x)\partial_{x_i}\quad\mbox{where}\quad h_i=\sum_{j=1}^n\Big(f_j\partial_{x_j}g_i-g_j\partial_{x_j}f_i\Big).
\end{equation*}
Let  $N\geq 2$ be an integer and let $X_1,\ldots ,X_N\in VF(\O)$. 
The Lie algebra generated by $\{X_i\}_{i=1}^N$,   Lie$(\{X_i\}_{i=1}^N)$,
is the smallest subspace of $VF(\O)$, containing $\{X_i\}_{i=1}^N$, which is invariant under the action of Lie brackets. So, we have that 
\begin{equation*}
\mbox{Lie}(\{X_i\}_{i=1}^N)=\bigcup_{k=1}^\infty \mbox{Lie}^k(\{X_i\}_{i=1}^N),
\end{equation*}
where $\mbox{Lie}^k(\{X_i\}_{i=1}^N)$ is defined recursively by taking 
$$ \mbox{Lie}^1(\{X_i\}_{i=1}^N)=\Span\,\{X_i\}_{i=1}^N$$
and, for $ k\ge 1$,
\begin{multline*}
 \mbox{Lie}^{k+1}(\{X_i\}_{i=1}^N)
 \\
 =\Span\Big(\mbox{Lie}^k(\{X_i\}_{i=1}^N)\cup\Big\{[X,X_j]~:~X\in\mbox{Lie}^k(\{X_i\}_{i=1}^N),\,j=1,\dots,N \Big\} \Big).
\end{multline*}
For all $x\in X$ we also set 
\begin{equation*}
\begin{cases}
 \mbox{Lie}^k(\{X_i\}_{i=1}^N)[x]=\big\{X(x)~:~X\in\mbox{Lie}^k(\{X_i\}_{i=1}^N)\big\}
 &
 \forall k\ge 1
 \vspace{.2cm}
 \\
\mbox{Lie}(\{X_i\}_{i=1}^N)[x]=\big\{X(x)~:~X\in\mbox{Lie}(\{X_i\}_{i=1}^N)\big\}.
\end{cases}
\end{equation*}
We say that $\{X_i\}_{i=1}^N\subset VF(\O)$ is {\em a
  system of H\"ormander vector fields} on $\O$ if  the following      bracket generating condition holds:
  \begin{equation}
\label{BGC}
 \mbox{Lie}(\{X_i\}_{i=1}^N)[x]=\R^n\qquad\forall x\in\O.
\end{equation}
Moreover, we  say that $\{X_i\}_{i=1}^N$ is  a
system of H\"ormander vector fields on $\bO$ if the bracket generating condition holds true on some open set $\O'\supset\bO$. Finally,  we say that
  $\{X_i\}_{i=1}^N$ is {\em strongly bracket generating} 
  on $\O$ if for every $v=(v_1,\ldots ,v_N)\in \R^n\setminus\{ 0\}$ 
 \begin{equation}\label{eq:SBGC}
 \Span \{ X_i\}_{i=1}^{N}[x]+\Span \left\{ \sum_{j=1}^N [X_j,X_i]\right\}_{i=1}^{N}[x]=\R^n\qquad\forall x\in\O.
\end{equation}
The following hypotheses $(H)$ will be assumed throughout:
\[
\begin{array}{ll}
 (H1)  & \mbox{\em $\Omega \subset\R^n$ is a bounded open set with
boundary $\Gamma$ of class $C^\infty$}, \\
 (H2)  &  \mbox{\em$\{X_1,\ldots ,X_N\}$ is  a
system of H\"ormander vector fields on $\bO$.}
\end{array}
\]
 Let us point out that, here, $X_1,\ldots, X_N$ need not be linearly independent, nor we  suppose  $N< n$.

\medskip
We define the {\em Hamiltonian} associated with $\{X_1,\ldots ,X_N\}$ by
\begin{equation}
\label{eq:h}
h(x,p)=\sum_{j=1}^N X_j(x,p)^2,\quad (x,p)\in T^*\Omega ',
\end{equation}
where we have set, with a slight abuse of notation,
\begin{equation*}
X_j(x,p):= \langle X_j(x),p\rangle\qquad\forall \,(x,p)\in T^*\Omega ',\,\forall j=1,\dots,N
\end{equation*}
(note that $ T^*\Omega '$ can be identified with $\Omega'\times \R^n$). 

The {\em characteristic set} of $\{X_1,\ldots ,X_N\}$ is given by
\begin{equation}
\label{eq:char}
\Char (X_1,\ldots X_N)=\big\{ (x,p)\in \Omega' \times (\R^n\setminus \{0\})~:~ h(x,p)=0\big\}.
\end{equation}

Finally, a point $x\in\Gamma$ is  called \emph{characteristic} if the linear space generated by $X_1(x),\ldots ,X_N(x)$ is contained in the tangent space to $\Gamma$ at $x$. We denote by $E\subset\Gamma$ the set of all  characteristic points. The following result, essential for this paper, is due to Derridj~\cite{D}.
\begin{theorem}
\label{t:derridj}
Under assumption $(H)$, $E$ is a closed subset of $\Gamma$
of $(n-1)$-dimensional Hausdorff measure zero.
\end{theorem}
Under assumption $(H)$, consider the Dirichlet problem
 \begin{equation}
\label{eq:sosq}
\left\{
\begin{array}{l}
\sum_{j=1}^N (X_jT)^2(x)=1\quad \text{ in
}\quad \Omega 
\vspace{.2cm}
\\
T=0 \quad \text{ on
}\quad \Gamma 
\end{array}
\right . 
\end{equation}
and observe that, since  the Hamiltonian $h(x,p)$ is not strictly convex in $p$,  characteristic points may appear.
%
%
%Let $T:\overline{\Omega}\to \R$ be the continuous viscosity 
%solution of 
%%
%
%
%From the PDE point of view, the above Dirichlet problem  has a typical feature: 
%
%
It is well known that \eqref{eq:sosq} admits a unique continuous
viscosity solution $T:\overline{\Omega}\to \R$. Indeed, taking  $\Gamma$ as the target set,
the minimum time function associated with $\{X_1,\ldots ,X_N\}$  
is a solution of the above equation. Such a function is defined as follows.  Given $x\in \bO$ and a measurable control 
$$u=(u_1,\dots,u_N):[0,+\infty[\to \R^N,
$$ 
taking values in the unit ball of $\R^N$, let us denote by  $y^{x,u}(\cdot)$ the unique solution of the Cauchy problem
 \begin{equation}\label{eq:se}
 \begin{cases}
 y'(t)=\sum_{j=1}^N u_j(t)X_j(y(t))
 &
 (t\geq 0)
 \\
 y(0)=x.
 \end{cases}
\end{equation}
Define the {\em transfer time} to $\Gamma$ as 
$$
\tau_\Gamma(x,u)=\inf \big\{ t\geq 0~:~y^{x,u}(t)\in \Gamma\big\}.
$$
Clearly,  $\tau_\Gamma(x,u)\in [0,+\infty]$. The  {\em Minimum Time Problem}  with target $\Gamma$ is the following:
\begin{itemize}
\item[(MTP)]  To minimize $\tau_\Gamma(x,u)$ over all controls  $u:[0,+\infty[\to \overline{B}_1(0)$.
\end{itemize}
Then, the {\em minimum time function}, defined as
$$
T(x)=\inf_{u(\cdot )}\tau_\Gamma(x,u)\qquad (x\in \overline{\Omega}),
$$
turns out to be the unique viscosity solution of the Dirichlet problem \eqref{eq:sosq} (see, e.g., \cite{BCD}).
It is well-known that H\"ormander's bracket ge\-ne\-ra\-ting condition implies that  \eqref{eq:se} is small time locally controllable. Hence, $T$ is  finite and continuous (see, e.g., \cite{BCD}).
\begin{remark}\em
We recall that a  $u(\cdot)$ is called an {\em optimal control} at a point $x\in\Omega$ if $T(x)=\tau_\Gamma (x,u)$. The corresponding solution of \eqref{eq:se}, $y^{x,u}$, is called the {\em time-optimal trajectory} at $x$ associated with $u$. 
%In more general terms, we will refer to an arc $x(\cdot)$ as a time-optimal trajectory at $x$ if $x(\cdot)=y^{x,u}(\cdot)$ for some optimal control  $u$.
\end{remark}
A more precise description of the continuity properties of $T$ can be given by looking at the maximal length of commutators needed to generate $\R^n$, that is, the function $k(\cdot)$ defined as follows:
\begin{equation}
\label{de:k(x)}
k(x)=\min\big\{k\ge 1~:~ \mbox{Lie}^kj(\{X_i\}_{i=1}^N)[x]=\R^n\big\} \qquad(x\in \Omega).
\end{equation}
We also define
\begin{equation}
  \label{eq:r}
  r_\O=\max\{ k(x)~:~x\in \overline{\Omega}\}.
\end{equation}
Then, thanks to the Baker-Campbell-Hausdorff formula (see, e.g., \cite{NSW}), $T$ turns out to be H\"older continuous of exponent $1/r_\O$  on $\overline{\Omega}$ .

Next, observe that the equation in \eqref{eq:sosq} can be re\-cast in the form
\begin{equation}\label{eq:ico}
\langle A(x)DT(x),DT(x)\rangle =1\quad \text{ in }\quad \Omega ,
\end{equation}
where $A(\cdot )$  is a suitable positive semidefinite $n\times n$ matrix with smooth entries and $DT=(\partial_{x_1}T,\dots,\partial_{x_n}T)$.
%We recall below the main known results for problem \eqref{eq:sosq} or, equivalently, for equation \eqref{eq:ico} with homogeneous Dirichlet boundary conditions. As we will see, the analysis  is well-developed when $A(\cdot )$ is positive definite but  few results are available  in the degenerate case. 
If $A(\cdot )$ is nondegenerate, then a classical viscosity argument ensures 
that any continuous
viscosity solution of \eqref{eq:ico}  
 is locally Lipschitz continuous in $\Omega$.
Furthermore, under very weak regularity conditions\footnote{It
  suffices to assume that $x\mapsto \langle A(x)p,p\rangle$ is
  semiconvex (i.e. it can be locally written as the sum of a convex
  with a smooth function) uniformly w.r.t. $p$ in the unit sphere of
  $\R^n$.}  on the entries of $A(\cdot )$, one can show that any continuous viscosity solution of \eqref{eq:ico} 
  is locally semiconcave in $\O$
(see \cite{A0}).
We recall that a function $T:\Omega\to \R$ is {\em locally
semiconcave} in $\Omega$ 
if for every $V\Subset \Omega$ there exists a constant $C$ such that
$D^2T\le CI$ in $\mathcal{D}'(V)$ (in the sense of quadratic forms).
 We point out that, for the aforementioned regularity results, the strict convexity of the map $p\mapsto \langle A(x)p,p\rangle$ is a crucial assumption. 

On the contrary, when $A(\cdot )$ in \eqref{eq:ico} is associated with a system of H\"ormander's vector fields with $N<n$,  known results mostly concern the continuity of solutions. For instance, it is known that any continuous viscosity solution of \eqref{eq:ico}  is H\"older continuous of exponent $1/r_\O$ given by \eqref{eq:r} (see, e.g., \cite{EJ}).
%On the other hand, when boundary conditions are added to \eqref{eq:ico}, a stronger result might be expected. For instance, in the example addressed in \cite{A1} one has that, taking
%%
%$$
%\Omega =\big\{ (x,y)\in \R^n\times\R~:~y>M|x|^{k+1}\big\}
%$$
%%
%(with $M>0$ and $k$ a positive integer), the nonnegative solution of
%%of the Dirichlet problem
%%
%$$
%\left\{
%\begin{array}{ll}
%|D_xT(x,y)|^2+|x|^{2k}(\partial_yT(x,y))^2=1\quad &\text{ in }\quad \Omega , 
%\\
%T=0\quad &\text{ on }\quad \partial\Omega,
%\end{array}
%\right .
%$$
%%
%where $D_xT(x,y)=(\partial_{x_1}T(x,y),\dots,\partial_{x_n}T(x,y))$, is locally Lipschitz continuous in $\Omega$ and H\"older continuous with exponent $\frac 1{k+1}$ just at $(0,0)$,  the only characteristic point of this system. In fact, we will show that this behaviour is typical in some sense, see
%
On the other hand, stronger regularity results are known under more restrictive assumption. For instance, as we recalled in the introduction, the following semiconcavity result holds for the sub-Riemannian setting~(\cite{CR}):  if  $X_1(x),\ldots ,X_N(x)$ 
are linearly independent for every $x\in \R^n$ and  strongly
bracket generating
 on $\R^n$, then for every $x_0\in\R^n$ the minimum time
function associated with the control system \eqref{eq:se} and target
$\{x_0\}$ is locally semiconcave---hence, locally Lipschitz---in $\R^n\setminus\{ x_0\}$ (see also Remark~\ref{re:Sect4} below).

\section{Singular time-optimal trajectories}\label{Sect:stot}
\setcounter{equation}{0}
\setcounter{theorem}{0}
\setcounter{proposition}{0}  
\setcounter{lemma}{0}
\setcounter{corollary}{0} 
\setcounter{definition}{0}

The  notion of singular trajectory that we give below plays a key role in our analysis.
For any boundary point $z\in \Gamma$ we denote by $\nu(z)$ the outward unit normal to
$\Gamma$ at $z$ and we set 
$$N_{\Gamma}(z) :=\{ \lambda\nu (z)~:~\lambda\ge 0\}.
 $$
 \begin{definition}\label{optimal_triple}
  Let $x\in \Omega$ and let $y(\cdot )=y^{x,u}(\cdot)$ be a time-optimal
  trajectory, with $u:[0,T(x)]\to \overline{B}_1(0)$. 
  We say that $y(\cdot )$ is {\em singular} if there exists an absolutely continuous
  arc 
  $p:[0,T(x)]\to \R^n \setminus \lbrace 0 \rbrace $ 
  such that, for a.e. $t\in [0,T(x)]$,
 \begin{equation}
    \label{eq:st}
    p_k'(t)=- \sum_{j=1}^N u_j(t) \langle \partial_{x_k}X_j(y(t)),
    p(t)\rangle , \quad \langle X_k(y(t)),p(t)\rangle =0,  
 \end{equation}
 for every $k=1,\ldots ,N$, and
 \begin{equation}
   \label{eq:transv}
p(T(x))\in N_{\Gamma}(y(T(x))).
\end{equation} 
\end{definition}
\begin{remark}\em
  (i) Our use of the adjective "singular" in Definition~\ref{optimal_triple} is classical but it is also motivated by Proposition~\ref{t:lipschitz-no-sing} below.

(ii) We recall that the strong bracket generating condition \eqref{eq:SBGC} ensures
the absence of singular time-optimal trajectories (see, e.g., \cite{CR}).

(iii)  We point out that, introducing the Control Theory Hamiltonian
  \begin{equation}
    \label{eq:ch}
H(x,p,u)=\sum_{j=1}^N u_j X_j(x,p),
  \end{equation}
the optimal triple $(y,u,p)$ arising from Definition \ref{optimal_triple} satisfies, for a.e. $t\in[0,T(x)]$, the Hamiltonian system
 \begin{equation}
    \label{eq:hs}
\left\{
    \begin{array}{l}
      y'(t)=D_pH(y(t),p(t),u(t))
      \vspace{.2cm}
      \\ 
      p'(t)=- D_xH(y(t),p(t),u(t))
       \vspace{.2cm}
%      \\ 
%      X_j(y(t),p(t))=0, \quad j=1,\ldots ,N.
    \end{array}
    \right .
    \end{equation}
In other words, a time-optimal trajectory is singular if it can be lifted in the phase space in such a way that the lifted trajectory:
\begin{itemize}
\item satisfies the
Hamiltonian system \eqref{eq:hs} (with Hamiltonian given by \eqref{eq:ch}) together with the transversality condition \eqref{eq:transv}, and 
\item  lies in the characteristic set $\Char (X_1,\ldots ,X_N)$.  
\end{itemize}
\end{remark}
The following characterization of  singular time-optimal
trajectories will be used throughout the paper.   
\begin{theorem}
\label{t:sincar}
Assume $(H)$. Let $x\in\Omega$ and let
$y^{x,u}$ be a time-optimal trajectory. Then, $y^{x,u}$ is singular if
and only if $y^{x,u}(T(x))\in E$.
\end{theorem}
\begin{proof}
Let $x\in \Omega$ and let $u(\cdot )$ a time-optimal control. Then, by the Pontryagin maximum principle  there exists an absolutely continuous function
$p:[0,T(x)]\to \R^n\setminus\{ 0\}$ (adjoint state) such that  the pair $(y^{x,u},p)$ satisfies the Hamiltonian system
\eqref{eq:hs} and the transversality condition   \eqref{eq:transv}.
%\begin{equation}
%\label{eq:transversality}
%p(T(x))\in N_{\Gamma}(y^{x,u}(T(x))) .
%\end{equation} 
%
%(We recall that  $N_{y}\Gamma =\{ \nu (y)\lambda\ |\ \lambda> 0\}$ and
%$\nu(y)$ is the ``exterior'' unit normal to
%$\Gamma$ at $y$.)  
Furthermore, for a.e. $t\in [0,T(x)]$, 
\begin{multline}
\label{eq:max}
H(y^{x,u}(t),p(t),u(t))=\max_{u\in \overline{B}_1(0) }\sum_{j=1}^N u_j \langle X_j(y^{x,u}(t)),p(t)\rangle 
\\
=\Big(\sum_{j=1}^N \big\langle X_j(y^{x,u}(t)),p(t)\big\rangle^2\Big)^{1/2}=\sqrt{h(y^{x,u}(t),p(t))}.
\end{multline}
Also, the function $[0,T(x)]\ni t\mapsto h(y^{x,u}(t),p(t))$ is constant. 

Now, suppose  $y^{x,u}(T(x))\notin E$. Then
$$
h(y^{x,u}(T(x)),p(T(x)))\neq 0  
$$
So, by the constancy of the Hamiltonian, 
$y^{x,u}$ cannot be singular.

Vice versa, assume  $y^{x,u}(T(x))\in E$. Then,  again by Pontryagin's principle, we deduce that $(y^{x,u}(t),p(t))\in \Char (X_1,\ldots, X_N)$, for every $t\in [0,T(x)]$, that is, $y^{x,u}$ is a singular time-optimal trajectory.
\end{proof}
The following is a point-wise notion of Lipschitz continuity (see also Federer \cite{F}).
\begin{definition}\label{de:lipschitz_at}
 We say that a function $f:\bO\to\R$ is Lipschitz  at a point $x_0\in\bO$ if there exists a neighbourhood $U$ of $x_0$ and a constant $L\ge 0$ such that 
 \begin{equation*}
|f(x)-f(x_0)|\le L|x-x_0|\qquad\forall x\in U\cap\bO.
\end{equation*}
\end{definition}
\noindent
Equivalently, $f$ is Lipschitz  at $x_0$ if and only if
\begin{equation*}
\limsup_{\bO\ni x\to x_0}\frac{|f(x)-f(x_0)|}{|x-x_0|}<\infty.
\end{equation*}
Observe that a function may well-be Lipschitz continuous at $x_0$ without being Lipschitz on any neighbourhood of $x_0$.
However, the interest of such kind of continuity is made clear by our next result. 
\begin{theorem}\label{t:nls}
  Assume $(H)$ and let $x_0\in \Omega$. 
  Then $T$ fails to be Lipschitz at $x_0$ if and only if there exists a singular time-optimal $y^{x_0,u}$ .
\end{theorem}
For the proof Theorem~\ref{t:nls} we borrow the following notion from nonsmooth analysis.
\begin{definition}[Proximal normals]
  Let $S\subset\R^n$ be a closed set. A vector $v\in\R^n$ is called a proximal normal to $S$ at $x$ if there exist  $\delta>0$ and $C>0$ (possibly depending on $x,v$) such that
  \begin{equation}
    \label{eqa:nc}
    \langle v, y-x\rangle \le C |y-x|^2\qquad \forall\,y\in B_\delta (x)\cap S. 
    \end{equation}
  The set of all proximal normals to $S$ at $x$ will be denoted by $N^P_S(x)$. 
  \end{definition}
We recall that the the hypograph of  a  function $f:\overline{\Omega}\to \R$ is the set
$$
\hp(f)=\big\{ (x,\alpha )\in \overline{\Omega}\times \R~:~\alpha \le f(x)\big\} .  
$$
\begin{definition}[Horizontal  proximal supergradients]
  The set of horizontal proximal subgradients of $f$ at a point $x\in\O$ is given by
  $$
\hps f(x)=\big\{ p \in \R^n~:~(-p,0)\in N_{\hp(f)}(x,f(x))\big\} . 
  $$
\end{definition}
We note that, in order for $f$ to be Lipschitz at a point $x_0\in \Omega$, it is necessary (but not sufficient) that $\hps f(x_0)=\{0\}$.

\begin{proof}[Proof of Theorem~\ref{t:nls}]

Assuming that $y^{x_0,u}$ is a singular time-optimal trajectory, let us show that  $\hps T(x_0)$ contains a nonzero vector, which  implies that $T$ cannot be Lipschitz at $x_0$.
Here we adapt, to systems of
H\"ormander’s vector fields, an argument used in \cite[Theorem~4.1]{CMN} to study
time optimal control problems for differential inclusions. We observe that the present context is different from the one in \cite{CMN} in that the nondegeneracy assumption of \cite{CMN} is not satisfied by our control system. 
We will use  the dual arc  $p(\cdot )$ for  $y^{x_0,u}$ provided by Pontryagin's principle.
Observe that, in view of \eqref{eq:transv},  $\exists\,\delta>0$ such that
    $$
    \overline{B}_\delta \big(y^{x_0,u}(T(x_0))+\delta \nu (x_0)\big)\cap \overline{\Omega} =\{ y^{x_0,u}(T(x_0))\},
       $$
or
       \begin{equation}\label{eqa:tran}
\big\langle p(T(x_0)), y-y^{x_0,u}(T(x_0))\big\rangle \le \frac 1 \delta \,| y-y^{x_0,u}(T(x_0))|^2\quad\forall y\in \overline{\Omega},
         \end{equation}
that is, $( p(T(x_0)),0)\in N_{\hp(T)} ( y^{x_0,u}(T(x_0)),0)$.

We now proceed  to show that
    \begin{equation}\label{eqa:fc}
     -p(0)\in \hps T(x_0), 
\end{equation}
    or, equivalently,
 $$
 (p(0),0)\in N_{\hp(T)} ( x_0, T(x_0)).
 $$
  We  have to prove that there exists a  constant $C_1>0$ such that, for every $x\in \overline{\Omega}$
           with $|T(x)-T(x_0)|<1$, we have that
           \begin{equation}
             \label{eqa:fingo}
             \langle p(0), x-x_0\rangle \le C_1 \big(|x-x_0|^2+(\alpha -T(x_0))^2\big) \quad \forall\,\alpha \le T(x). 
             \end{equation}
  Set $T_0=T(x_0)$ and $T_1=T(x)$. We shall analyse the two  cases $T_1\le T_0$ and $T_1>T_0$ separately.

\medskip\noindent
{\em Case 1: }\fbox{$T_1\le T_0$}

\noindent
It suffices to verify \eqref{eqa:fingo} with $\alpha =T_1$.
Set
$$
\ox_0=y^{x_0,u}(T_1)\quad\mbox{and}\quad \ox=y^{x,u}(T_1).
$$
Notice that $T(\ox_0)=T_0-T_1\in [0,1[$ and $\ox \in\overline{\Omega}$. 
    Furthermore
    \begin{equation}
      \label{eqa:esttra}
|y^{x,u}(t)-y^{x_0,u}(t)|\le C_1 |x-x_0|\quad\forall\, t\in[0,T_1]
    \end{equation}
    for a suitable constant $C_1>0$ depending only on $\Omega$, the $C^1$ norms of the vector fields $X_1,\ldots ,X_N$, and $\max_{\overline{\Omega}}T$.
    Then, we find that
    \begin{multline}
      \label{eqa:par}
\langle p(0),x-x_0\rangle = \langle p(T_1) , \ox-\ox_0\rangle 
\\
-\int_0^{T_1} \frac d{dt} \big\langle p(t), y^{x,u}(t)-y^{x_0,u}(t)\big\rangle\, dt . 
    \end{multline}
  We now claim that, for every $t\in [0,T_1]$, 
\begin{equation}
\label{eq:claim_variation}
 \Big|\frac d{dt} \big\langle p(t),y^{x,u}(t)-y^{x_0,u}(t)\big\rangle \Big|\le C_2 |p(t)| \,|y(t)-y^{x,u}(t)|^2,
\end{equation}
for a suitable constant $C_2$ depending on $\Omega$, the $C^2$ norms of the vector fields $X_1,\ldots ,X_N$, and $\max_{\overline{\Omega}}T$.
Indeed, recalling \eqref{eq:st} we deduce that
\begin{multline*}
\frac d{dt} \big\langle p(t), y^{x,u}(t)-y^{x_0,u}(t)\big\rangle
\\
= \sum_{j=1}^N u_j(t) \Big\langle p(t),X_j(y^{x,u}(t))-X_j(y^{x_0,u}(t))\Big\rangle 
\\
-\sum_{j=1}^N u_j(t) \Big\langle p(t),dX_j(y^{x_0,u}(t)))( y^{x,u}(t)-y^{x_0,u}(t))\Big\rangle .
\end{multline*}
Hence, \eqref{eq:claim_variation} follows.

Then, recalling \eqref{eq:claim_variation} , \eqref{eqa:par}, and \eqref{eqa:esttra}, we find
\begin{equation}
  \label{eqa:par2}
\langle p(0),x-x_0\rangle \le C_3 |x-x_0|^2 + \langle p(T_1) , \ox-\ox_0\rangle.
\end{equation}
In order to bound the last term in the above inequality
we observe that
\begin{multline*}
  \langle p(T_1) ,\ox-\ox_0\rangle =\langle p(T_0) , \ox-\ox_0\rangle + \langle p(T_1)-p(T_0) , \ox-\ox_0\rangle 
  \\
\le \langle p(T_0) , \ox-\ox_0\rangle +C_4 ( |T_1-T_0|^2+|\ox-\ox_0|^2) 
\end{multline*}
for a suitable constant $C_4$.
Now, recall that
$$
 \langle p(T_0) , \ox-\ox_0\rangle  =\langle p(T_0) , y^{x,u}(T_1)-y^{x_0,u}(T_1)\rangle
$$
Then, we have
\begin{eqnarray*}
\lefteqn{ \langle p(T_0) , \ox-\ox_0\rangle} 
\\
&=&  \langle p(T_0) , y^{x_0,u}(T_0)-y^{x_0,u}(T_1)\rangle +\langle p(T_0) , y^{x,u}(T_1)-y^{x_0,u}(T_0)\rangle
\end{eqnarray*}
The second term can be estimated by \eqref{eqa:tran}:
\begin{multline*}
\langle p(T_0) , y^{x,u}(T_1)-y^{x_0,u}(T_0)\rangle
\\
 \le \frac 1\delta |y^{x,u}(T_1)-y^{x_0,u}(T_0)|^2\le
C_5 \big(|x-x_0|^2+(T_1-T_0)^2\big)
\end{multline*}
for some constant $C_5>0$.
Moreover, we have that 
$$
\langle p(T_0) , y^{x_0,u}(T_0)-y^{x_0,u}(T_1)\rangle =\sum_{j=1}^N\int_{T_1}^{T_0} u_j(t)\langle p(T_0), X_j(y^{x_0,u}(t))\rangle dt .
$$
Recalling that $y^{x_0,u}(T_0)\in E$, i.e. 
$$
\langle p(T_0),X_j(y^{x_0,u}(T_0))\rangle=0\quad \text{for}\quad j=1,\ldots ,N,
$$
we deduce that 
\begin{eqnarray*}
\lefteqn{\langle p(T_0) , y^{x_0,u}(T_0)-y^{x_0,u}(T_1)\rangle}
\\
&=& \sum_{j=1}^N \int_{T_1}^{T_0} u_j(t)\langle p(T_0),X_j(y^{x_0,u}(t))-X_j(y^{x_0,u}(T_0)\rangle\, dt \le C_6( T_1-T)^2
\end{eqnarray*}
for some constant $C_6>0$. Thus, \eqref{eqa:fingo} follows. 

\medskip\noindent
{\em Case 2:} \fbox{$T_1> T_0$}

\noindent
We have that
\begin{eqnarray*}
\lefteqn{\langle p(0),x-x_0\rangle}
\\
&= &\langle p(T_0),y^{x,u}(T_0)-y^{x_0,u}(T_0)\rangle -\int_0^{T_0} \frac d{dt}\langle p(t),y^{x,u}(t)-y^{x_0,u}(t)\rangle dt .
\end{eqnarray*}
The second term on the right-hand side of the last identity can be bounded by a constant times $|x-x_0|^2$ like in \eqref{eq:claim_variation}. In order to estimate the remaining term it suffices to observe that $y^{x,u}(T_0)\in\bO$. So, \eqref{eqa:tran} and \eqref{eqa:esttra} yield
$$
\langle p(T_0), y^{x,u}(T_0)-y^{x_0,u}(T_0)\rangle\le C_7 |x-x_0|^2.
$$
This shows that, if $y^{x_0,u}$ is a singular time-optimal trajectory, then $T$ fails to be Lipschitz at $x_0$.

In order to prove the converse, we argue by contradiction by assuming that  $T$ fails to be Lipschitz continuous at $x_0$ and that all optimal trajectories from $x_0$ are not singular. Let  $y^{x_0,u}$ be any time-optimal trajectory. 
So, $y^{x_0,u}(T(x_0))\in \Gamma\setminus E$ by of  Theorem~\ref{t:sincar}.
Then, appealing to Theorem~\ref{t:derridj} we can find  $\delta>0$ such that
$
B_\delta (y^{x_0,u}(T(x_0)))\cap E=\emptyset. 
      $
Furtheremore, taking $\delta$ small enough, we may assume that there exists a positive constant $C$ such that
\begin{equation}
  \label{eq:distest}
T(y)\le C \dist (y,\Gamma )\quad \text{ for every }\quad y\in B_\delta (y^{x_0,u}(T(x_0)))\cap \overline{\Omega} .
\end{equation}
Since $T$ is not Lipschitz  at $x_0$, there exists $\{x_j\}\subset\Omega$ such that
\begin{equation}
  \label{eq:abs}
\frac{|T(x_j)-T(x_0)|}{|x_j-x_0|}\geq j \text{ and }x_j\to x_0, \text{ as }j\to \infty . 
  \end{equation}
Set
$$
T_0=T(x_0) \quad\mbox{and}\quad T_j=T(x_j).$$
Two possibilities my occur: either $T_0\le T_j$ or $T_0>T_j$.
Observe that, in both  cases,
\begin{equation}
  \label{eq:esttra}
|y^{x_0,u}(t)-y^{x_j,u}(t)|\le C_1 |x_0-x_j| \quad \forall \,t \in [0, \min\{ T ,T_j\}],  
\end{equation}
where $C_1$ is a positive constant depending on $\Omega$, the $C^1$ norms of the vector fields $X_1,\ldots, X_N$, and $\max_\Omega T$. 

\medskip\noindent
{\em Case 1:} \fbox{$T_0\le T_j$} 

\noindent
We have that
\begin{equation}
  \label{eq:esj}
j\le \frac{T_j-T_0}{|x_j-x_0|}\le \frac {T(y^{x_j,u}(T_0)) }{|x_j-x_0|}. 
\end{equation}
Since $y^{x_0,u}(T_0)\in \Gamma $, by \eqref{eq:esttra} we deduce that, for $j$ large enough, 
$$
y^{x_j,u}(T_0)\in \bO \cap B_\delta (y^{x_0,u}(T_0)).
$$
Then, by \eqref{eq:esj} and \eqref{eq:distest}, we find
$$
j\le C\,\frac{ \dist (y^{x_j,u}(T_0) ,\Gamma ) }{|x_j-x_0|}\le C\, \frac{ |y^{x_j,u}(T_0)-y^{x_0,u}(T_0) }{|x_j-x_0|}. 
$$
So, using  \eqref{eq:esttra} once again, we find the contradiction
$
j\le CC_1.
$

\medskip\noindent
{\em Case 2:} \fbox{$T_j< T_0$} 

\noindent
Denote by $u_j$ a time-optimal control at $x_j$.   
After extending $u_j$ to $[0,T_0]$ by setting $u_j(t)\equiv 0$ for $t\in [T_j,T_0]$, we choose a
subsequence, still labeled $\{x_j\}$, such that $y^{x_j,u_j}$ converges uniformly to $y^{x_0,u_0}$ on all $[0,T_0]$, for a suitable $u_0(\cdot )$.  
Furthermore, because of $T_j\uparrow T_0$, we may assume that
$y^{x_j,u_j}(T_j)\in \Gamma\cap B_\delta (y^{x_0,u_0}(T_0))$ for $j$ large enough.
Then,  by dynamic programming, \eqref{eq:abs} yields
$$
j\le \frac{T(y^{x_0,u_j}(T_j))}{|x_0-x_j|}. 
$$
Using \eqref{eq:esttra} (with $u$ replaced by $u_j$) and \eqref{eq:distest} we find once more the contradiction 
$$
j\le C \frac{|y^{x_0,u_j}(T_j)- y^{x_j,u_j}(T_j)|}{|x_0-x_j|}\le CC_1 . 
$$
This implies that $T$ cannot be Lipschitz continuous along $y^{x_0,u_0}$.

In order to show that $y^{x_0,u_0}$ is singular  observe that otherwise $T$ would be Lipschitz continuous on a neighborhood of $y^{x_0,u_0}(T(x_0))$. Then Lipschitz regularity would propagate backwards along $y^{x_0,u_0}$, thus contradicting what we have just shown.  
\end{proof}
Our next result is a straightforward consequence of Theorem~\ref{t:nls} and dynamic programming.
\begin{corollary}\label{t:lipschitz-no-sing}
Assume $(H)$. Let $x_0\in \Omega$ and let $y^{x_0,u}$ be a singular time-optimal trajectory. Then, for any $t\in[0,T(x_0)[$,  $T$  fails to be Lipschitz continuous at $y^{x_0,u}(t)$. 
\end{corollary}

\section{Regularity}\label{Sect:reg}
\setcounter{equation}{0}
\setcounter{theorem}{0}
\setcounter{proposition}{0}  
\setcounter{lemma}{0}
\setcounter{corollary}{0} 
\setcounter{definition}{0}

In this section we study the interior and boundary regularity of the solution of equation \eqref{eq:sosq}. 
Here is the main theorem of this section.
\begin{theorem}[Interior regularity] 
  \label{t:semiconcave}
Under assumption $(H)$, the following properties are equivalent:
  \begin{enumerate}
    \item  {\em (MTP)} admits no singular time-optimal
  trajectory;
    \item 
      $T$ is locally semiconcave in $\Omega$;

      \item $T$ is
      locally Lipschitz continuous in $\Omega$.  
\end{enumerate}
  \end{theorem} 
\begin{remark}\label{re:Sect4} \em
\begin{itemize}
\item[(i)] The fact that the existence of singular optimal trajectories may destroy
the regularity of a solution of a first order Hamilton-Jacobi equation
was already observed by Sussmann (in an implicit form) in \cite{S}
and (explicitly) by Agrachev in \cite{Ag}. The regularity these authors
consider is  subanalyticity of the point-to-point distance function
associated with real analytic distributions. The aforementioned
subanalyticity results were extended to solutions of the Dirichlet
problem in \cite{Tre}.
\item[(ii)] We observe that, as shown in \cite{CJT},  there exists an open dense set $\mathcal{O}$ in $(VF(\Omega))^N$ (equipped with the $C^\infty$ topology) such that condition (1) in Theorem~\ref{t:semiconcave} holds whenever $(X_1,\ldots , X_N)\in \mathcal{O}$. 
\item[(iii)] In \cite{CR}, the local semiconcavity of the sub-riemannian distance to a point $x_0$, $d_{SR}(x_0,x)$, is proved without giving any  estimate of the dependance on $x_0$ of the semiconcavity constant of $d_{SR}(x_0,\cdot)$. Therefore, the semiconcavity of 
$$T(x)=\min_{y\in \Gamma}d_{SR}(y,x)$$ 
 is not a direct consequence of the result of \cite{CR}.
\end{itemize}
\end{remark}
The proof of Theorem~\ref{t:semiconcave} is based on methods from optimal control. 
\begin{proof}
It is well-known that the local semiconcavity in $\Omega$ of $T$ yields the local  Lipschitz continuity of $T$ in the same set. Thus, in order to prove the theorem, it suffices to show that the following  implications hold true:
\begin{itemize}
\item[(i)] if  (MTP) admits  no singular time-optimal trajectories then $T$ is locally semiconcave in $\Omega$; 
\item[(ii)]  if $T$ is locally Lipschitz continuous in $\Omega$ then (MTP) admits  no singular time-optimal trajectories. 
\end{itemize} 

Now, (ii) follows from Corollary~\ref{t:lipschitz-no-sing}. We therefore proceed to prove (i). First of all, let us recall that for a smooth controlled system the minimum time function is locally semiconcave whenever the target is a noncharacteristic smooth compact manifold. This result can be proved arguing as in \cite{CSi1}. For the reader convenience we give  a  sketch of the proof. It is well-known that since the target $\Gamma$ is a smooth manifold, the Euclidean distance function is smooth on a neighborhood of $\Gamma$. Furthermore, since $\Gamma$ is a noncharacteristic hypersurface, the minimum time function can be estimated from above by the Euclidean distance function in some neighborhood of $\Gamma$ (see \cite[Prop. 2.2]{CSi1}). Owing to these two facts,  the minimum time function inherits the semiconcavity of the distance function in a neighborhood of $\Gamma$. Finally, by standard techniques which use time-optimal trajectories to ``propagate'' regularity, one can show that the semiconcavity of $T$ holds true even away from the target. 

Now, let $x_0\in \Omega$ and consider all the time-optimal trajectories starting at $x_0$. We denote by $\lbrace u_\alpha \rbrace $, for $\alpha$ running in a suitable set of indices $\mathcal{A}(x_0)$, the set of all the optimal controls at  $x_0$. By Theorem~\ref{t:sincar}, we have that $y^{x_0,u_\alpha}(T(x_0))\notin E$, for every $\alpha\in \mathcal{A}$. We claim that there exist a smooth compact submanifold $\Gamma_0\subset \Gamma$, with $\Gamma_0\cap E=\emptyset$,  and a neighborhood of $x_0$, $V\Subset \Omega$, such that 
\begin{equation}
\label{eq:K}
\quad y^{x,u_\alpha}(T(x)) \in \Gamma_0,\quad\forall x\in V\,,\; \forall \,\alpha\in\mathcal{A}(x).
\end{equation}
Indeed, thanks to Theorem \ref{t:derridj}, for any  $x_0\in\Omega$ there exists a smooth compact submanifold $\Gamma_0\subset \Gamma$  such that
\begin{equation}
  \label{eq:k0}
  \begin{cases}
 \; \Gamma_0\cap E=\emptyset 
 \\
 \;y^{x_0,u_\alpha}(T(x_0)) \in \Gamma_0,\quad \forall \alpha\in\mathcal{A}(x_0)
\\
\inf \big\{ | y^{x_0,u_\alpha}(T(x_0))-y|~:~y\in \Gamma\setminus \Gamma_0,\, \alpha\in\mathcal{A}(x_0)\big\} >0.
\end{cases}
\end{equation}
Arguing by contradiction, let us assume that there exist a sequence $x_j\in \Omega$, converging to $x_0$, and a sequence of optimal controls $u_j\in\mathcal{A}(x_j)$ such that $y^{x_j,u_j}(T(x_j))\notin \Gamma_0$. Then, possibly taking a subsequence, we deduce that there exists $\bar{u}\in \mathcal{A}(x_0)$ such that
$$
\lim_{j\to\infty} y^{x_j,u_j}(T(x_j))=y^{x_0,\bar{u}}(T(x_0))
$$
in contradiction with \eqref{eq:k0}. Thus, \eqref{eq:K} holds true.

This shows that the restriction of $T$ to the set $V$ coincides with the restriction to $V$ of the minimum time function 
for the control system \eqref{eq:se} with target $\Gamma_0$. Since $\Gamma_0$ is a smooth compact manifold, by the first part of this proof we deduce that $T$ is semiconcave in $V$. A compactness argument then shows that $T$ is semiconcave on every  set $C\Subset \Omega$, thus completing
the proof of (i).
\end{proof}
Next, we give an example showing that  singular time-optimal trajectories may well occur for the problem of interest to this paper.
\begin{example}
  \label{p:example}
  In $\R^3$ consider  vector fields
  $$
X_1=\partial_{x_1},\qquad X_2= (1-x_1)\partial_{x_2}+x_1^2 \partial_{x_3} .
  $$
Then, there exists a bounded open  set $\Omega \subset \R^3$, with $C^\infty$ boundary, such that the viscosity solution of the Dirichlet problem
  \begin{equation}\label{eq:sl}
    \left\{
    \begin{array}{ll}
(X_1T)^2+ (X_2T)^2=1\quad &\text{ in }\quad \Omega ,
      \\ \\
      T|_{\Gamma}=0,
    \end{array}
    \right .
    \end{equation}
fails to be locally Lipschitz in $\Omega$. 
\end{example}
\begin{proof}

For fixed $a>0$,  consider  the subset of $\R^3$
  $$
\Gamma_a=\big\{ (x_1, x_2, -(x_2-a)^2)~:~x_1\in \R , x_2>0\big\} 
  $$
and the controlled system
\begin{equation}
  \label{eq:contsys}
\left\{
\begin{array}{ll}
  \dot{y}_1=u_1 
  \\
  \dot{y}_2 = u_2 \,(1-y_1)
  \\
  \dot{y}_3 = u_2 \,y_1^2
  \end{array} 
  \right . 
\end{equation}
with initial condition $y(0)=(0,0,0)$.
We denote by $T_{\Gamma_a}$ the minimum time function associated with  system \eqref{eq:contsys} and target $\Gamma_a$.

We claim that, for $a<1$, $T_{\Gamma_a}(0)=a$ and $(0,t,0)$, $t\in [0,a]$, is a singular time-optimal trajectory.
Indeed:
  \begin{itemize}
    \item taking $p(t)=(0,0,-1)$, 
      $y(t)$ can be lifted to a solution of \eqref{eq:hs} (in the characteristic set of the vector fields) satisfying the transversality condition \eqref{eq:transv};
    \item $y(t)$ reaches $\Gamma_a$ in time $a$;
    \item in order to reach the target, starting at the origin, $ \dot{y}_2$ should be positive whilst  $\dot{y}_3$ should be negative (this can be done provided that $1-y_1<0$ but the minimal time, in order to have $1-y_1<0$, is larger than $1$.) 
  \end{itemize}
  Hence, for any $a\in ]0,1[$ we have that $T_{\Gamma_a}(0)=a$.

    In order to complete the proof it remains to define the set $\Omega$.
    In $\R^3$, consider the Euclidean open ball with center at $(0,a,0)$ and radius  $a/2$ and denote such a ball by $B$.
    Then, for every $b\in ]a/2,a[$, 
    $$T_{\Gamma_a}(0,b,0)=T_{\Gamma_a\cap B}(0,b,0)=a-b.$$ 
    It is clear that the hypersurface $\Gamma_a\cap B$ can be smoothly extended outside $B$ and  we can construct a $C^\infty$  hypersurface $\Gamma$ such that
    $\Gamma =\partial\Omega$ (for a suitable open bounded set $\Omega$). For $b\in ]a/2,a[$ small enough we have that $(0,b,0)\in \Omega$ 
    and $T_{\Gamma}(0,b,0)=a-b$. As noted above, this implies that the (MTP) associated with the controlled system \eqref{eq:contsys} and target $\Gamma$ admits a singular time-optimal trajectory. Then, owing to Theorem~\ref{t:semiconcave}, $T_{\Gamma}$ is not a locally Lipschitz function.  
\end{proof}
We complete this section with the study of the regularity of $T$ at the boundary of $\O$. As we did for Lipschitz continuity, we begin by introducing a point-wise notion of H\"older regularity.
\begin{definition}\label{de:holder_at}
 We say that a function $f:\bO\to\R$ is H\"older continuous of exponent $\alpha\in ]0,1]$ at a point $x_0\in\O$ if there exist a neighborhood $U\subset\O$ of $x_0$ and a constant $C\ge 0$ such that 
 \begin{equation*}
|f(x)-f(x_0)|\le C|x-x_0|^\alpha\qquad\forall x\in U\cap\bO.
\end{equation*}

\end{definition}
The following result shows that, for the Dirichlet problem,  the maximal length of commutators $k(\cdot)$ gives the``optimal'' H\"older exponent  at characteristic boundary points. 
\begin{theorem}[Boundary regularity] 
  \label{t:bregularity}
Assume $(H)$. Then:
  \begin{enumerate}
 \item for any $x\in \Gamma \setminus E$,   
  $T$ is $C^\infty$ in a neighborhood of $x$;
\item for any $x\in E$,  
  $T$ is H\"older (not Lipschitz)  continuous at $x$ of
  exponent $1/k(x)$, with $k(x)$ given by \eqref{de:k(x)}.  
\end{enumerate}
 \end{theorem} 
\noindent
Even though part of the above conclusions may be known to the reader,  for completeness we provide a proof of Theorem~\ref{t:bregularity} in Appendix~\ref{a:b}.

In order to state a more general regularity result, we introduce the set
$$
\sing_L T =\Big\{ x\in \overline{\O}~:~\limsup_{\bO\ni y\to x}\frac{|T(y)-T(x)|}{|y-x|}=\infty \Big\}
$$
which consists of all points at which $T$ fails to be Lipschitz.%
\begin{proposition}\label{t:sing-l}
  Assume  $(H)$ and let $T$ be the viscosity solution of \eqref{eq:sosq}.
  Then, $\sing_LT$ is a closed set. 
  \end{proposition} 
\begin{proof}
Let $x_j\in \sing_L T$, $j\in \N$, be a sequence of points converging to $x\in \O$. We claim that $x\in \sing_L T$. Indeed, since $T$ fails to be Lipschitz at $x_j$, we have that there exist control functions $u_j$ such that:
  \begin{enumerate}
 % \item $y^{x_j,u_j}$ is a time-optimal trajectory;
    \item $y^{x_j,u_j}$ is a singular time-optimal trajectory (this is a consequence of Theorem~\ref{t:nls}); 
\item possibly taking a subsequence, $y^{x_j,u_j}$ converges uniformly to an optimal trajectory $y^{x,u}$. 
  \end{enumerate}
  We claim that $y^{x,u}$ is singular. 
  Indeed, Theorem~\ref{t:derridj} implies that $E$ is a closed set and $y^{x_j,u_j}(T(x_j))\in E$, by Theorem~\ref{t:sincar}.
  Hence, we find that 
  $$
  E\ni \lim_{j\to \infty }y^{x_j,u_j}(T(x_j))=y^{x,u}(T(x)).
  $$
So, by using once again Theorem~\ref{t:sincar}, we deduce that $y^{x,u}$ is a singular time-optimal trajectory. 

In order to complete the proof we observe that, if $x\in \Gamma$, then
$$
x=\lim_{j\to \infty}y^{x_j,u_j}(T(x_j))\in E
$$
and the conclusion follows by Theorem~\ref{t:bregularity}(2).
\end{proof} 
Our next result fully describes the regularity of the solution to \eqref{eq:sosq}.
\begin{theorem}\label{t:general-t}
  Assume $(H)$ and let $T$ be the viscosity solution of  \eqref{eq:sosq}.
  Then, $T$ fails to be Lipschitz at all points of  the closed set $\sing_L T$ and is semiconcave on the relatively open set $\overline{\O }\setminus \sing_L T$. 
  \end{theorem} 
\noindent
The proof of Theorem~\ref{t:general-t} is similar to the one of Theorem~\ref{t:semiconcave} and is omitted. 
\begin{remark}\label{t:sslip}\em
An immediate consequence of the above result is that $\sing_L T\setminus E$ coincides with the Lipschitz singular support of $T$. Indeed,  $x\in \Omega\setminus \sing_L T$ if and only if there exists a neighborhood of $x$, $V$, such that $T$ is Lipschitz continuous on $V$. 
  \end{remark}
\subsection{Sufficient conditions for  regularity}\label{Sect:suff}
In this section, we give conditions to guarantee the absence of
singular time-optimal trajectories. Let us recall that, in canonical coordinates\footnote{More in general a symplectic form in a $C^\infty$ manifold is a non-degenerate, closed $C^\infty$ two form.}, the symplectic form
in $T^*\Omega$ is the $2$-form  
\begin{equation}
\label{eq:sigma}
\sigma =\sum_{k=1}^n dp_k\wedge dx_k .
\end{equation}
Furthermore, for any $\rho\in T^*\Omega$,  given a vector space $W\subset
T_\rho (T^*\Omega)$ we denote by $W^\sigma$ the symplectic orthogonal
to $W$, i.e. 
$$
W^\sigma =\{ v\in T_\rho (T^*\Omega)~:~\sigma (v,w)=0\,,\; \forall w\in W\}. 
$$
Finally, we say that a manifold $M\subset T^*\Omega$ is \emph{symplectic} if
the restriction of $\sigma$ to $M$ is nondegenerate, i.e. 
\begin{equation*}
T_\rho M \cap (T_\rho M)^\sigma =\{ 0\}\qquad\forall \,\rho\in M.
\end{equation*}
We have the following result.
\begin{theorem}
  \label{t:nstot}
 Assume (H). Then {\em (MTP)}  admits no singular time-optimal
  trajectory if any of the following conditions is satisfied:
  \begin{itemize}
\item[(i)] $\Gamma$ is noncharacteristic (i.e. $E=\emptyset$);  
    \item[(ii)]  $\Char (X_1,\ldots ,X_N)$ is a symplectic manifold. 
\end{itemize}
  \end{theorem}
\begin{proof}
For noncharacteristic $\Gamma$  the conclusion follows from Theorem~\ref{t:sincar}.

Let us assume that $\Char(X_1,\dots,X_N)$ is a symplectic manifold.
We have that 
$$
\Span \{ dX_1(x,p),\ldots , dX_N(x,p)\} \perp T_{(x,p)} \Char (X_1,\ldots ,X_N) 
$$
(here ``$\perp$'' stands for  orthogonality w.r.t. the Euclidean scalar product).  
Hence, we find 
$$
(\Span \{ dX_1(x,p),\ldots , dX_N(x,p)\})^\sigma \subset (T_{(x,p)} \Char (X_1,\ldots ,X_N))^{\sigma} .   
$$
On the other hand, since $\Char (X_1,\ldots ,X_N)$ is a symplectic manifold, using  the identity
$$
T_{(x,p)} \Char (X_1,\ldots ,X_N)\cap (T_{(x,p)} \Char (X_1,\ldots ,X_N))^{\sigma} =\{ 0\},
$$
we conclude that either the broken Hamiltonian flow \eqref{eq:hs}  is transveral to $ \Char (X_1,\ldots ,X_N)$, or it is a stationary flow, i.e., there are no singular time-optimal trajectories. This completes our proof.
\end{proof}
\begin{remark}\label{rem:sub}\em
 As a consequence of our results and \cite[Theorem~3.4]{Tre}, assuming that $\{X_1,\ldots ,X_N\}$ are real analytic and  $\partial\Omega$ is a real analytic submanifold, we have that, if $\Char (X_1,\ldots ,X_N)$ is a symplectic manifold, then $T$ is a subanalytic function in $\Omega$ (see, e.g. \cite{Ta}, for the definition and  basic properties of  subanalytic sets and functions). 
\end{remark}

\subsection{Examples} \label{Sect:examples}
In this section, we discuss problems to which Theorem~\ref{t:nstot} can be directly applied as well as examples where the assumptions of Theorem~\ref{t:nstot} are not fulfilled, but one can still use Theorem~\ref{t:semiconcave} to study the regularity of $T$. 

In the example below we consider  Heisenberg vector fields and use Theorem \ref{t:nstot} to show that this problem admits no singular time-optimal trajectory. It is worth no\-ting that the same conclusion  follows from the fact that such vector fields are strongly bracket generating. 
\begin{example}[Heisenberg vector fields] 
\label{p:heisenberg}
 In $\R^3$ consider vector fields
$$
X_1=\partial_{x_1},\quad  X_2=\partial_{x_2}+x_1\partial_{x_3}
$$
and let $\Omega$ be a bounded open  set with $C^\infty$ boundary. 
We have that
$$
\Char (X_1,X_2)=\big\{ (x_1,x_2,x_3,0,-x_1 p_3,p_3):(x_1,x_2,x_3)\in \Omega , \,p_3\neq 0\big\}
$$
is a smooth submanifold of $\R^6$ of codimension $2$. Furthermore, the restriction of $\sigma$ to $\Char (X_1,X_2)$ is nondegenerate, i.e. $\Char (X_1,X_2)$ is a symplectic manifold. Then, by Theorem~\ref{t:nstot} (ii), we conclude that  (MTP) has no singular time-optimal trajectory.

\end{example}

Since we consider a boundary value problem for the eikonal equation,
due to the interaction of the boundary of $\Omega$
with $\Char (X_1,\ldots ,X_N)$, even systems of vector fields admitting,
in general, singular time-optimal trajectories may have a better
behaviour when $\Omega$ enjoys specific properties. In this regard, we consider once more the system of vector fields
in Example~\ref{p:example}. We point out that such a system,  
introduced in \cite{LS}, admits, in general, strictly abnormal geodesics---in the language of the geometric control theory. We will show that, whenever $\Omega$ is convex, we can exclude the occurrence of singular time-optimal trajectories. 
\begin{example}
\label{p:gmartinet}
In $\R^3$ consider vector fields
$$
X_1=\partial_{x_1},\quad  X_2=(1-x_1)\partial_{x_2}+x_1^2\partial_{x_3}
$$
and let $\Omega$ be a bounded  convex  open set with $C^\infty$ boundary. Then,  (MTP)
admits no singular time-optimal trajectory. 
\end{example}
\begin{proof} We argue by contradiction assuming the existence of a singular time-optimal trajectory.
Let us write the two systems appearing in \eqref{eq:hs} in this specific case. We have
\begin{equation}
  \label{eq:hsmx}
  \left\{
  \begin{array}{l}
    \dot{y}_1=u_1 \\
    \dot{y}_2=u_2 (1-y_1) \\
    \dot{y}_3=u_2y_1^2  
  \end{array}
  \right.
  \quad 
  \left\{
  \begin{array}{l}
    \dot{p}_1= u_2 (p_2-2y_1 p_3) \\
    \dot{p}_2=0 \\
    \dot{p}_3=0. 
  \end{array}
  \right .
  \end{equation}
The only possible initial point $\bar{x}\in \Omega$  of  a singular time-optimal trajectory must be of the form $(0,\bar{x}_2,\bar{x}_3)$, $\bar{x}_{2}, \bar{x}_{3}\in\mathbb{R}$. Then  the corresponding 
solution of \eqref{eq:hsmx} is given by 
$$
(y(t),p(t)):=  \Big(0,\bar{x}_2+\int_0^t u_2(s), \bar{x}_3 ,0,0,\bar{p}_3\Big)
$$
where $(0,\bar{x}_2, \bar{x}_3, 0,0,\bar{p}_3)$ are the ``initial'' conditions for the Hamiltonian system in
\eqref{eq:hsmx}. Let $x_2^{*}$ be such that
$
y(T)=(0,x_2^*,\bar{x}_3 )
$
and note that $(0,x_2^*,\bar{x}_3 )\in \Gamma$ is a characteristic point.
Since $\Omega$ is a  smooth convex set, there exists a smooth function $\Phi =\Phi (x_1,x_2,x_3)$ such that
\begin{equation*}
 \overline{\Omega}\cap B_\delta ( 0,x^*_2,\bar{x}_3)=\big\{ (x_1,x_2,x_3)\in  B_\delta ( 0,x^*_2,\bar{x}_3):\Phi (x_1,x_2,x_3)\le 0\big\}
\end{equation*}
and
\begin{multline*}
  \partial \Omega \cap B_\delta ( 0,x^*_2,\bar{x}_3)
  \\
  =\big\{ (x_1,x_2,x_3)\in  B_\delta ( 0,x^*_2,\bar{x}_3):\Phi (x_1,x_2,x_3)= 0\,,\;D\Phi (x_1,x_2,x_3)\not= 0\big\}
\end{multline*}
for a suitable $\delta >0$. Now, the transversality condition \eqref{eq:transv} in Pontryagin's maximum principle implies that 
$
D\Phi (y(T))=\lambda (0,0,1),
$
for some $\lambda\not= 0$. In particular, we deduce that 
$
\partial_{x_3}\Phi (y(T))\not=0. 
$
By appealing to the implicit function theorem and convexity assumption, we can suppose that, near $(0,x^*_2, \bar{x}_3)$, $\Omega$ is the hypograph of a smooth  concave function $(x_1,x_2)\mapsto \phi (x_1 , x_2)$, which attains its maximum at $(0,x^*_2)$ and such that $\phi (0,x^*_2)= \bar{x}_3$. This implies that $(0,x^*_2, \bar{x}_3)$ cannot be the final point of a trajectory starting from $\O$ because the line segment $(0,x^*_2+s,\bar{x}_3)$, $s\in \R$, can at most belong to the boundary of $\O$. We have thus reached a contradiction.
\end{proof}

As shown in the next example, in some cases, even if the characteristic set is not a symplectic manifold but it can be splitted into a disjoint union of symplectic submanifolds, our approach can be applied. 
\begin{example}
Let $\Omega\subset \R^3$ be a
bounded open  set with smooth boundary and let $k$ be a positive integer. In $\R^3$, consider vector fields
$$X_1=\partial_{x_1}-x_2^{2k+1}\partial_{x_3}\quad\mbox{and}\quad
X_2=\partial_{x_2}+x_1^{2k+1}\partial_{x_3}.$$ 
Then no singular time-optimal trajectories exists.
\end{example}
\begin{proof}
The characteristic set is given by
\begin{multline*}
  \Char (X_1,X_2)
  \\
  =\big\{ (x_1,x_2,x_3, x_2^{2k+1}p_3 ,  -x_1^{2k+1}p_3 ,p_3 )~:~x_1,x_2,x_3\in \R , \, p_3 \not= 0\big\}.
\end{multline*}
Therefore,  $ \Char (X_1,X_2)$ can be split into the connected submanifolds 
\begin{multline*}
\Sigma_{1,\pm} =\big\{ (x_1,x_2,x_3, x_2^{2k+1}p_3 ,  -x_1^{2k+1}p_3 ,p_3 )~:
\\
 x_1,x_2,x_3\in \R , \; (x_1,x_2)\not=(0,0),\; \pm p_3 > 0\big\}
\end{multline*}
and
$$
\Sigma_{2,\pm} =\big\{ (0,0,x_3, 0 ,  0 ,p_3 )~:~ x_3\in \R , \; \pm p_3 > 0\big\}.
$$
Moreover, all these submanifolds are symplectic (the rank of the symplectic form is constant and the symplectic form is nondegenerate on these sets).
So, there is no singular time-optimal trajectory. 
\end{proof}

\noindent
Observe that, in the above example, the characteristic set is a manifold of condimension $2$ in $\R^3$, but the rank of the symplectic form is not constant, i.e., $\Char (X_1,X_2)$ is not a symplectic manifold.  Then, since we are not assuming that $E=\emptyset$, none of the conditions in Theorem~\ref{t:nstot} is satisfied.

We complete this section with an application of Theorem~\ref{t:general-t}.
\begin{example}
  Let $\Omega\subset $ be a bounded open  set with smooth boundary and consider  vector fields
  $$
X_1=\partial_{x_1},\quad X_2=(1-x_1)\partial_{x_2}+x_1^2\partial_{x_3}.
  $$
Then, $\sing_L T$ is a set of measure zero and  $T$ has a second order Taylor expansion at a.e. point of $\Omega$.   
\end{example} 
\begin{proof}
  We observe that, if we show that $\sing_L T$ is a set of measure zero, then $T$ turns out to be semiconcave in $\Omega \setminus \sing_L T$  by Theorem~\ref{t:general-t}. So, the conclusion follows by the Alexandrov theorem on twice differentiability of a convex function.
  Hence, let us show that $\sing_L T$ is a set of measure zero.
  We observe that all singular time-optimal trajectories can be lifted to the characteristic set
  \begin{multline*}
\Char (X_1,X_2)=\big\{ (x_1,x_2,x_3, 0, p_2,p_3)~:~(x_1,x_2,x_3)\in \Omega ,
\\(1-x_1)p_2+x_1^2p_3=0,\quad p_3\not= 0\big\} . 
 \end{multline*}
On the other hand, the characteristic set can be decomposed as follows
$$
\Char (X_1,X_2)=V_1\cup V_2 \cup V_3, 
$$
with 
\begin{multline*}
V_1=\big\{ (x_1,x_2,x_3, 0, p_2,p_3)~:~(x_1,x_2,x_3)\in \Omega ,
\\
(1-x_1)p_2+x_1^2p_3=0\quad -p_2+2x_1p_3\not=0\big\},
\end{multline*}
$$
V_2=\big\{(2,x_2,x_3,0,4p_3,p_3)~:~(2,x_2,x_3)\in\Omega,\quad  p_3\not= 0  \big\},
$$
and
$$
V_3=\big\{ (0,x_2,x_3,0,0,p_3)~:~(0,x_2,x_3)\in \Omega ,\quad p_3\not= 0\big\} .  
$$
We observe that the only solutions of \eqref{eq:hs}  in the set $V_1$  are stationary points (it is a symplectic manifold).
Hence, a nontrivial lifting of a singular time-optimal trajectory can only belong  to the set $V_2\cup V_3$.
By Theorem~\ref{t:nls}, we deduce that 
$$
\sing_L T\subset \big\{ (x_1,x_2,x_3)\in \Omega~:~x_1=0\quad\text{ or }\quad x_1=2 \big\}  
$$
and the conclusion follows. 
  \end{proof}
 \appendix

\section{Proof of Theorem~\ref{t:bregularity}}\label{a:b}
\setcounter{equation}{0}
\setcounter{theorem}{0}
\setcounter{proposition}{0}  
\setcounter{lemma}{0}
\setcounter{corollary}{0} 
\setcounter{definition}{0}

The fact that $T$ is of class $C^\infty$ in a neighborhood of every non-characteristic point can be proved arguing as follows.
Let $x_0\in \Gamma \setminus E$. Then, by Theorem~\ref{t:derridj}, there exists a neighborhood of $x_0$ in $\Omega'$, $W$, such that $W\cap E=\emptyset$. Consider the Dirichlet problem
$$
h(x,Dv)=1\quad\text{ in }\quad W\cap \Omega 
$$
with initial condition $v=0$ in $\Gamma \cap W$.
This is a noncharacteristic Dirichlet problem with smooth data. Then, possibly replacing $W$ with a smaller neighborhood of $x_0$,
we have that there exists a unique solution $v\in C^\infty (W\cap \Omega)$. Since the characteristics of the above problem are time-optimal trajectories, we conclude that $T\in C^\infty (W\cap \Omega)$.

Let us now prove H\"older regularity at a characteristic boundary point. Take $x_0\in E$ and denote by $T_{\{x_0\}}$ the minimum time function for the controlled system \eqref{eq:se} with target $\{ x_0\}$. Let $W$ be a neighborhood of $x_0$ in $\Omega'$  such that
$$
\max_{ x\in W\cap \overline{\Omega}} k(x)=k(x_0). 
$$ 
We want to show that $T$ is H\"older continuous of exponent $1/k(x_0)$ at $x_0$.
It is well known~(\cite[Theorem~1.15, Chap.~IV]{BCD}) that there exists a  constant $C>0$ such that 
$$
T_{\{x_0\}} (x)\le C |x-x_0|^{ 1/k(x_0)}, \quad \forall x\in W. 
$$
By possibly taking a smaller set $W$, we may assume that 
$$
T_{\{ x_0\} }(x)\geq T(x), \quad \forall x\in \partial (W\cap \Omega )
$$
and 
$$
h(x,DT_{\{ x_0\} }(x))=1=h(x,DT(x))\quad \text{ in }\quad W\cap \Omega 
$$
(in the viscosity sense).
So, the comparison principle implies that
$$
T(x)\le T_{\{ x_0\} }(x) \quad \forall x\in W\cap \Omega .  
$$
Then, we deduce that
$$
|T(x)-T(x_0)|=T(x)\le C  |x-x_0|^{ 1/k(x_0)}, \quad \forall x\in W\cap \Omega .
$$
Let us show that, if $x_0\in E$, then $T$ cannot be Lipschitz continuous at $x_0$. We argue by contradiction and assume   $\exists\,C>0$ such that
\begin{equation}
  \label{eq:nl}
  T(x)\le C|x-x_0| \quad \forall x\in \overline{\Omega}\cap B_\delta (x_0), 
  \end{equation}
for some $\delta >0$.
Then, by Theorem~\ref{t:derridj}, there exists a sequence $x_j\in \Gamma\cap B_\delta (x_0)\setminus E$ such that $x_j\to x_0$ as $j\to \infty$. Then
$$
DT(x_j)= \langle DT(x_j), \nu(x_j)\rangle \nu(x_j). 
$$
By \eqref{eq:nl}, $\{\langle DT(x_j), \nu(x_j)\rangle\}_j$ is a bounded sequence. Hence, possibly taking a subsequence, we find that, as $j\to\infty$, 
$$
\langle DT(x_j), \nu(x_j)\rangle \to c \quad \text{ and }\quad \nu (x_j)\to \nu (x_0),  
$$
for some $c\in\R$.
Thus,  we find the contradiction
$$
1=h(x_j,DT(x_j))\; \text{ and }\; \lim_{j\to\infty} h(x_j,DT(x_j))= c^2 h(x_0,\nu (x_0))=0.
$$

\end{document}